\documentclass[a4paper, marginparwidth=15mm]{amsart}

\usepackage{amsthm}
\usepackage{amssymb}
\usepackage{dsfont}
\usepackage{enumitem}
\usepackage{stmaryrd}
\usepackage{colonequals}
\usepackage[mathcal]{euscript}
\usepackage{tikz-cd} 
\usepackage[all]{xy}
\SelectTips{cm}{}
\usepackage{microtype}
\usepackage{mathtools}

\usepackage[top=35mm, bottom=35mm, outer=30mm, inner=30mm, marginparwidth=20mm]{geometry}

\usepackage{hyperref}
\hypersetup{%
  bookmarksnumbered=true,%
  colorlinks=true,%
  linkcolor=[RGB]{0, 0, 153},
  citecolor=[RGB]{0,153,51},      
  urlcolor=[RGB]{153,0,0},
  filecolor=blue,%
  menucolor=blue,%
  bookmarksopen=true,%
  bookmarksdepth=2,%
  pageanchor=true,
  breaklinks=true,
  pdfusetitle=true}

\usepackage[nameinlink,capitalize]{cleveref}
\crefname{equation}{}{}
\usepackage[backend=bibtex, style=alphabetic, maxnames=50]{biblatex}
\renewbibmacro{in:}{}
\makeatletter
\@namedef{subjclassname@2020}{\textup{2020} Mathematics Subject Classification}
\makeatother

\numberwithin{equation}{section}

\theoremstyle{plain}
\newtheorem{theorem}[equation]{Theorem}
\newtheorem*{theorem*}{Theorem}
\newtheorem{proposition}[equation]{Proposition}
\newtheorem{lemma}[equation]{Lemma} 
\newtheorem{corollary}[equation]{Corollary}
\newtheorem{conjecture}[equation]{Conjecture}
\newtheorem*{conjecture*}{Conjecture}

\theoremstyle{definition}
\newtheorem{definition}[equation]{Definition}

\theoremstyle{remark}
\newtheorem{remark}[equation]{Remark} 

\newtheorem*{ack}{Acknowledgements}

\newcommand*{\intref}[2]{\def\tmp{#1}\ifx\tmp\empty\hyperref[#2]{\ref*{#2}}\else\hyperref[#2]{#1~\ref*{#2}}\fi}

\renewcommand{\dim}{\operatorname{dim}}

\newcommand{\Ext}{\operatorname{Ext}}
\newcommand{\ext}{\operatorname{ext}}

\newcommand{\Hom}{\operatorname{Hom}}

\newcommand{\Pic}{\operatorname{Pic}}
\newcommand{\prfdg}{\mathcal{P}\!\mathit{erf}\!\operatorname{--}}

\newcommand{\rk}{\operatorname{rk}}

\newcommand{\Spec}{\operatorname{Spec}}

\newcommand{\mcA}{\mathcal{A}}
\newcommand{\mcB}{\mathcal{B}}

\newcommand{\mcO}{\mathcal{O}}

\newcommand{\mcT}{\mathcal{T}}

\newcommand{\sfD}{\mathsf D}

\newcommand{\sfK}{\mathsf K}

\newcommand{\bbC}{\mathbb C}

\newcommand{\bbP}{\mathbb P}
 
\newcommand{\bbZ}{\mathbb Z}

\title[A Phantom on a Rational Surface]{A Phantom on a Rational Surface}

\author[Johannes~Krah]{Johannes Krah}
\address{Fakult\"at f\"ur Mathematik, Universit\"at Bielefeld, D-33501 Bielefeld, Germany
}
\email{jkrah@math.uni-bielefeld.de}
\thanks{The research was funded by the Deutsche Forschungsgemeinschaft (DFG, German Research Foundation) -- SFB-TRR 358/1 2023 -- 491392403}

\begin{document}

\begin{abstract}
	We construct a non-full exceptional collection of maximal length consisting of line bundles on the blow-up of the projective plane in $10$ points in general position. This provides a counterexample to a conjecture of Kuznetsov and to a conjecture of Orlov.
\end{abstract}



\maketitle
\section{Introduction}
Let $X$ be a smooth projective variety over the field of complex numbers and denote by $\sfD^b(X)$ the bounded derived category of coherent sheaves on $X$. A nontrivial admissible subcategory $\mcA \subseteq \sfD^b(X)$ is called a \emph{phantom} if the Grothendieck group $\sfK_0(\mcA)$ vanishes.
The first example of a phantom was constructed by Gorchinskiy--Orlov \cite{gorchinskiy_orlov_geometric_phantom_categories}.
We provide the first example of a variety which admits a full exceptional collection and a phantom subcategory.
\begin{theorem}
	\label{thm:main_thm}
	Let $X$ be the blow-up of $\bbP^2_\bbC$ in $10$ closed points $p_1,\dots, p_{10}\in \bbP^2_\bbC$ in general position. Denote by $H$ the divisor class obtained by pulling back the class of a hyperplane in $\bbP^2_\bbC$ and denote by $E_i$ the class of the exceptional divisor over the point $p_i$, $1\leq i \leq 10$. 
	Then
	\begin{align}\label{eq:the_collection_introduction}
		&\langle \mcO_X, \mcO_X(D_1), \dots, \mcO_X(D_{10}), \mcO_X(F), \mcO_X(2F) \rangle \subseteq \sfD^b(X), \tag{1} \\
		\text{where}\quad &D_i  \coloneqq -6H + 2\sum_{j=1}^{10} E_j - E_i \quad \text{and} \quad	F  \coloneqq -19H+6\sum_{i=1}^{10}E_i, \nonumber
	\end{align}
	is an exceptional collection of maximal length which is not full.
\end{theorem}
It was previously shown in \cite[Thm.~6.35]{pirozhkov_admissible_subcategories_of_del_pezzo_surfaces_new} that a del Pezzo surface $Y$ does not admit a phantom in $\sfD^b(Y)$. Moreover, we showed in earlier work \cite{krah_mutationa_of_num_exc_collections_on_surfaces} that on the blow-up of $\bbP^2_\bbC$ in $9$ points in very general position every exceptional collection of maximal length consisting of line bundles is full. We discovered the exceptional collection \cref{eq:the_collection_introduction} while trying to increase the number of blown up points in \cite[Thm.~1.3]{krah_mutationa_of_num_exc_collections_on_surfaces}.

Any blow-up of $\bbP^2_\bbC$ in a finite set of points admits a full exceptional collection, \cref{thm:main_thm} disproves the following conjecture of Kuznetsov:
\begin{conjecture}[{\cite[Conj.~1.10]{kuznetsov_semiorthogonal_decompositions_in_algebraic_geometry}}]
	Let $\mcT=\langle E_1, \dots, E_n\rangle$ be a triangulated category generated by an exceptional collection. Then any exceptional collection of length $n$ in $\mcT$ is full.
\end{conjecture}
As a consequence, the right- or left-orthogonal complement of the collection is a phantom category. Hence, this disproves the following conjecture of Orlov:
\begin{conjecture}[{\cite[Conj.~3.7]{orlov_finite_dim_dg_algebras_and_their_geometric_realizations}}] 
	There are no phantoms of the form $\prfdg \mathcal{R}$, where $\mathcal{R}$ is a smooth finite-dimensional dg-algebra and $\prfdg \mathcal{R}$ is the dg-category of perfect dg-modules over $\mathcal{R}$.
\end{conjecture}
Note that if $\mcA$ is an admissible subcategory of $\sfD^b(X)$ and $\sfD^b(X)$ admits a full exceptional collection, then by \cite[Cor.~3.4]{orlov_finite_dim_dg_algebras_and_their_geometric_realizations} $\mcA$ has a dg-enhancement quasi-equivalent to $\prfdg \mathcal{R}$, where $\mathcal{R}$ is a smooth finite-dimensional dg-algebra.

Recently, Chang--Haiden--Schroll gave an example of a triangulated category admitting a full exceptional collection such that the braid group action by mutations does not act transitively on the set of full exceptional collections up to shifts \cite{chang_haiden_schroll_braid_group_actions_on_branched_coverings_and_full_exceptional_sequences}.
Since mutations of exceptional collections do not change the generated subcategory, our example provides a surface where the braid group does not act transitively on the set of exceptional collections of maximal length.

As mentioned to us by Kalck, our example further shows the existence of a triangulated category which admits a tilting object and a presilting object which is not a direct summand of a silting object. Indeed the full exceptional collection $(\mcO_X, \mcO_X(E_1), \dots, \mcO_X(E_{10}), \mcO_X(H), \mcO_X(2H))$ is strong, so that $\sfD^b(X)$ admits a tilting object.
Moreover, the object
$$P\coloneqq \mcO_X \oplus \mcO_X(D_1)[2] \oplus \dots \oplus \mcO_X(D_{10})[2] \oplus \mcO_X(F)[4] \oplus \mcO_X(2F)[6] \in \sfD^b(X),$$
is presilting since it satisfies $\Hom_{\sfD^b(X)} (P,P[i])=0$ for all $i>0$, but not silting since $P$ does not generate $\sfD^b(X)$. Finally, since $\sfD^b(X)$ is a Krull--Schmidt category, one deduces from \cite[Thm.~2.27]{aihara_iyama_silting_mutations_in_triangulated_categories} that $P$ is not the direct summand of a silting object.
Another such example was previously constructed in \cite{liu_zhou_a_negative_answer_to_complement_question_for_presilting_complexes}, but our example has the additional property that the number of non-isomorphic indecomposable summands of $P$ is equal to $\rk \sfK_0(X)$.

\begin{ack}
	This work is part of the author's dissertation, supervised by Charles Vial whom we wish to thank for helpful discussions and explanations. We discovered the existence of the exceptional collection \cref{eq:the_collection_introduction} in the context of our previous work \cite{krah_mutationa_of_num_exc_collections_on_surfaces}, where we study the transitivity of the braid group action on (numerically) exceptional collections on surfaces using a classification obtained by Vial in \cite{vial_exceptional_collections_and_the_neron_severi_lattice_for_surfaces}.
\end{ack}

\section{Exceptional Collections}
We recall the basic definitions and properties of exceptional collections and semiorthogonal decompositions. For a detailed reference we refer to \cite{kuznetsov_semiorthogonal_decompositions_in_algebraic_geometry} and the references therein.

Let $X$ be a smooth projective variety over $\bbC$ and denote by $\sfD^b(X)$ the bounded derived category of coherent sheaves on $X$.
A \emph{semiorthogonal decomposition} of $\sfD^b(X)$ is an ordered collection $(\mcA_1, \dots, \mcA_n )$ of full triangulated subcategories such that
\begin{equation*}
	\Hom_{\sfD^b(X)}(A_i, A_j)=0 \; \text{for all} \; A_i \in \mcA_i, A_j \in \mcA_j, j<i
\end{equation*}
and the smallest triangulated subcategory of $\sfD^b(X)$ containing $\mcA_1, \dots, \mcA_n$ is $\sfD^b(X)$. We write 
$$\sfD^b(X)=\langle \mcA_1, \dots, \mcA_n \rangle$$
for such a semiorthogonal decomposition.
A full triangulated subcategory $\mcA \subseteq \sfD^b(X)$ is called \emph{admissible} if the inclusion functor $\mcA \hookrightarrow \sfD^b(X)$ admits both a right and a left adjoint. Such an admissible subcategory gives rise to the semiorthogonal decompositions $\sfD^b(X)=\langle \mcA^\perp, \mcA \rangle = \langle \mcA, {^\perp \mcA} \rangle$, where
\begin{align*}
	^\perp \mcA &\coloneqq \{F \in \sfD^b(X) \mid \Hom_{\sfD^b(X)}(F, A)=0 \; \text{for all} \; A \in \mcA\} \\
	\text{and} \; \mcA^\perp &\coloneqq \{F \in \sfD^b(X) \mid \Hom_{\sfD^b(X)}(A, F)=0 \; \text{for all} \; A \in \mcA\}
\end{align*}
are the \emph{left-} and \emph{right-orthogonal complements} of $\mcA$. If $\mcA$ is admissible, so are $^\perp \mcA$ and $\mcA ^\perp$.
An object $E \in \sfD^b(X)$ is called \emph{exceptional} if $\Hom_{\sfD^b(X)}(E,E)= \bbC$ and $\Hom_{\sfD^b(X)}(E,E[k])=0$ for all $k\neq 0$.
A collection $(E_1, \dots, E_n)$ of exceptional objects is called an \emph{exceptional collection} if
$$\sfD^b(X)=\langle \mcA,  E_1, \dots,  E_n \rangle,$$
where $\mcA = \langle E_1, \dots, E_n \rangle ^\perp$ is a semiorthogonal decomposition. Note that, by abuse, we denote the subcategory generated by the object $E_i$ also by $E_i$. An exceptional collection $(E_1, \dots, E_n)$ is \emph{full} if $\mcA=0$.
A subcategory generated by an exceptional collection $\langle E_1, \dots, E_n \rangle$ is always admissible; in particular, its left- and right-orthogonal complements are again admissible.

A semiorthogonal decomposition $\sfD^b(X)=\langle \mcA_1, \dots, \mcA_n \rangle$ yields a direct sum decomposition of the Grothendieck group of $\sfD^b(X)$:
$$ \sfK_0(X)=\sfK_0(\mcA_1) \oplus \dots \oplus \sfK_0(\mcA_n).$$
An exceptional collection $(E_1, \dots, E_n)$ is of \emph{maximal length} if there exists no further exceptional object $F \in \sfD^b(X)$ such that $(E_1, \dots, E_n, F)$ is an exceptional collection.
Because $\langle E_i \rangle \cong \sfD^b(\Spec \bbC)$ for an exceptional object $E_i$, we have $\sfK_0(\langle E_i \rangle) =\bbZ[E_i]$. Thus, if $\sfK_0(X)$ is finitely generated as an abelian group and $n  = \rk \sfK_0(X)$, then any exceptional collection of length $n$ is of maximal length.

Assume that $\sfK_0(X)$ is finitely generated and $(E_1, \dots E_n)$ is an exceptional collection of length $n=\rk \sfK_0(X)$. The additivity of $\sfK_0$ among semiorthogonal decompositions implies that $\sfK_0(\mcA)=\mathrm{tors}(\sfK_0(X))$ is a finite group, where $\mcA = \langle E_1, \dots, E_n \rangle ^\perp$. If $\mcA \subseteq \sfD^b(X)$ is a nonzero admissible subcategory with finite $\sfK_0(\mcA)$, then by definition $\mcA$ is a \emph{quasi phantom} and if additionally $\sfK_0(\mcA)=0$, then $\mcA$ is called a \emph{phantom}.

Let $\mcA \subseteq \sfD^b(X)$ and $\mcB \subseteq \sfD^b(Y)$ be triangulated full subcategories. Then $\mcA \boxtimes \mcB \subseteq \sfD^b(X\times Y)$ denotes the smallest triangulated full subcategory of $\sfD^b(X\times Y)$ which is closed under direct summands and contains all objects of the form $p_X^*A\otimes^{\mathrm{L}} p_Y^*B$ for $A \in \mcA$ and $B \in \mcB$.
Following \cite[Def.~1.9]{gorchinskiy_orlov_geometric_phantom_categories} an admissible subcategory $\mcA \subseteq \sfD^b(X)$ is called a \emph{universal phantom} if for all smooth projective varieties $Y$ the category $\mcA \boxtimes \sfD^b(Y)$ is a phantom.

\section{SHGH Conjecture}\label{se:shgh_conj}
Let $X$ be the blow-up of the projective plane $\bbP^2_\bbC$ in a set of closed points $p_1, \dots, p_n \in \bbP^2_\bbC$. Denote by $E_i \subseteq X$ the $(-1)$-curve over the point $p_i$ and recall that $\Pic(X)=\bbZ H \oplus \bbZ E_1 \oplus \dots \oplus \bbZ E_n$, where $H$ is the pullback of a hyperplane in $\bbP^2_\bbC$.
The class of a divisor $D$ on $X$ can be uniquely written as
$$D=dH-\sum_{i=1}^n m_i E_i$$
for some $d, m_i \in \bbZ$.
Moreover, the intersection product satisfies $H^2=1$, $E_i^2=-1$, $H\cdot E_i=0$, and $E_i \cdot E_j=0$ for all $i \neq j$.
If $d>0$ and $m_i \geq 0$, the space of global sections $H^0(X, \mcO_X(D))$ can be identified with the space of homogeneous polynomials $P \in \bbC[X, Y,Z]$ of degree $d$ such that $P$ vanishes to order $\geq m_i$ at $p_i$. If the points are chosen in general position, meaning that $h^0(D)\coloneqq \dim H^0(X, \mcO_X(D))$ is minimal, then the following conjecture due to Segre--Harbourne--Gimigliano--Hirschowitz predicts the value of $h^0(D)$.
\begin{conjecture}[SHGH]
	Let $X$ be the blow-up of $\bbP^2_\bbC$ in $n$ points in general position and let $d>0, m_i\geq 0$ be integers. Let $D\coloneqq dH-\sum_{i=1}^n m_iE_i$, then 
	\begin{equation*}
		\dim H^0(X, \mcO_X(D)) = \max (0, \chi(X, \mcO_X(D)))
	\end{equation*}
	or there exists a $(-1)$-curve $C\subseteq X$ such that $C\cdot D \leq -2$.
\end{conjecture}
If the divisor $D=dH-\sum_{i=1}^n m_iE_i$ is in \emph{standard form}, i.e.\ $d\geq m_1\geq \dots \geq m_n$ and $d-m_1-m_2-m_3 \geq 0$, it is known that there exist no $(-1)$-curve $C$ such that $D\cdot C \leq -2$; see, e.g., \cite[Prop.~1.4]{ciliberto_miranda_homogeneous_interpolation_on_ten_points}.
The SHGH Conjecture is known to be true in various cases of low multiplicity. For example in \cite{dumnicki_jarnicki_new_effective_bounds_on_the_dimension_of_a_linear_system_in_p2} the conjecture is verified for divisors with all multiplicities $m_i \leq 11$.
Alternatively, for a single explicit divisor $D$ it is possible to compute the actual value of $h^0(D)$ using a computer.
We will use these computations of $h^0(D)$ to show that the collection in \cref{thm:main_thm} is exceptional.

\section{Height and Pseudoheight of Exceptional Collections}
Kuznetsov introduced in \cite{kuznetsov_height_of_exceptional_collections_and_hochschild_cohomology_of_quasiphantom_categories} the so-called \emph{height} of an exceptional collection $\langle E_1, \dots, E_n \rangle \subseteq \sfD^b(X)$:
If $\mathcal{D}$ is a smooth and proper dg-category and $\mathcal{B} \subseteq \mathcal{D}$ a dg-subcategory, Kuznetsov defines the \emph{normal Hochschild cohomology} $\mathrm{NHH}^\bullet(\mathcal{B}, \mathcal{D})$ of $\mathcal{B}$ in $\mathcal{D}$ as a certain dg-module \cite[Def.~3.2]{kuznetsov_height_of_exceptional_collections_and_hochschild_cohomology_of_quasiphantom_categories}. The height of an exceptional collection $(E_1, \dots, E_n)$ is then defined as
$$\mathrm{h}(E_1, \dots, E_n)\coloneqq \min\{k \in \bbZ \mid \mathrm{NHH}^k(\mathcal{E}, \mathcal{D})\neq 0 \}$$
where $\mathcal{D}$ is a dg-enhancement of $\sfD^b(X)$ and $\mathcal{E}$ the dg-subcategory of $\mathcal{D}$ generated by the exceptional objects $(E_1, \dots, E_n)$.
In general, the normal Hochschild cohomology $\mathrm{NHH}^\bullet (\mathcal{E}, \mathcal{D})$ can be computed using a spectral sequence \cite[Prop.~3.7]{kuznetsov_height_of_exceptional_collections_and_hochschild_cohomology_of_quasiphantom_categories}.
For our purpose it will be sufficient to consider a coarser invariant of an exceptional collection, the so-called \emph{pseudoheight}.
\begin{definition}[{\cite[Def.~4.4, Def.~4.9]{kuznetsov_height_of_exceptional_collections_and_hochschild_cohomology_of_quasiphantom_categories}}]
	For any two objects $F, F' \in \sfD^b(X)$ define the \emph{relative height} as
	$$\mathrm{e}(F, F')\coloneqq  \inf \{k \in \bbZ \mid \Ext^k(F, F') \neq 0 \}.$$
	For an exceptional collection $(E_1, \dots, E_n)$ the \emph{pseudoheight} is
	\begin{equation*}
		\mathrm{ph}(E_1, \dots, E_n)\coloneqq  \min_{1\leq a_0 <\dots <a_p \leq n}  \left (  \mathrm{e}(E_{a_0}, E_{a_1}) + \dots + \mathrm{e}(E_{a_{p-1}}, E_{a_p}) +\mathrm{e}(E_{a_p}, S^{-1}(E_{a_0})) - p \right),
	\end{equation*}
	where $S=-\otimes \omega_X[\dim X]$ is the Serre functor of $\sfD^b(X)$.
	The \emph{anticanonical pseudoheight} is
	\begin{equation*}
		\mathrm{ph}_{\mathrm{ac}}(E_1, \dots, E_n)\coloneqq  \min_{1\leq a_0 <\dots <a_p \leq n}  \left (  \mathrm{e}(E_{a_0}, E_{a_1}) + \dots + \mathrm{e}(E_{a_{p-1}}, E_{a_p}) +\mathrm{e}(E_{a_p}, E_{a_0}\otimes \omega_X^{-1}) - p \right).
	\end{equation*}
\end{definition}
Clearly, $\mathrm{ph}_{\mathrm{ac}} = \mathrm{ph}- \dim X$.
\begin{lemma}[{\cite[Lem.~4.5]{kuznetsov_height_of_exceptional_collections_and_hochschild_cohomology_of_quasiphantom_categories}}]\label{lem:height_bigger_than_pseudoheight}
	For an exceptional collection $(E_1, \dots, E_n )$ in $\sfD^b(X)$ we have $\mathrm{h}(E_1, \dots, E_n) \geq \mathrm{ph}(E_1, \dots, E_n)$.
\end{lemma}
We will use the following criterion to show that the exceptional collection in \cref{thm:main_thm} is not full.
\begin{proposition}[{\cite[Prop.~6.1]{kuznetsov_height_of_exceptional_collections_and_hochschild_cohomology_of_quasiphantom_categories}}]\label{prop:height_zero_if_coll_full}
	Let $X$ be a smooth projective variety and $(E_1, \dots, E_n)$ in $\sfD^b(X)$ an exceptional collection. If $\mathrm{h}(E_1, \dots, E_n) >0$, then $( E_1, \dots, E_n)$ is not full.
\end{proposition}
In particular, if $\mathrm{ph}_{\mathrm{ac}}(E_1, \dots, E_n)>-\dim X$, then the collection is not full.

\section{Proof of \texorpdfstring{\cref{thm:main_thm}}{Theorem \ref{thm:main_thm}}}\label{sec:proof_of_thm}
Let $X$ be the blow-up of $\bbP^2_\bbC$ in $10$ points in general position.
A full exceptional collection in $\sfD^b(X)$ consisting of line bundles is given by
$$\sfD^b(X)=\langle \mcO_X, \mcO_X(E_1) , \dots, \mcO_X(E_{10}), \mcO_X(H), \mcO_X(2H) \rangle.$$
Since the canonical class $K_X=-3H+\sum_{i=1}^{10}E_i$ satisfies $K_X^2=-1$, the Picard lattice admits an orthogonal decomposition $\Pic(X)=K_X^\perp \oplus \bbZ K_X$ and one can compute that a basis of $K_X^\perp$ is given by $H-E_1- E_2-E_3, E_1-E_2, \dots, E_9-E_{10}$.
Consider the orthogonal transformation $\iota\colon \Pic(X) \to \Pic(X)$ which multiplies an element of $K_X^\perp$ by $-1$ and is the identity on $\bbZ K_X$.
We compute
\begin{equation*}
	D_i  \coloneqq \iota(E_i)=-6H + 2\sum_{j=1}^{10} E_j - E_i \quad \text{and} \quad
	F  \coloneqq \iota(H)=-19H+6\sum_{i=1}^{10}E_i.
\end{equation*} 
Since $\iota$ fixes the canonical class, one can deduce from the Riemann--Roch formula that
\begin{equation}\label{eq:interesting_collection} \tag{1}
	(\mcO_X, \mcO_X(D_1) , \dots, \mcO_X(D_{10}), \mcO_X(F), \mcO_X(2F))
\end{equation}
is a \emph{numerically exceptional collection}, i.e.\ it is semiorthogonal with respect to the Euler pairing 
$$\chi(F, G)\coloneqq \sum_{i \in \bbZ} (-1)^{i} \dim \Hom_{\sfD^b(X)}(F, G[i]) \; \text{for all} \; F, G \in \sfD^b(X) $$
and each object $F$ in the collection satisfies $\chi(F, F)=1$.
Moreover, it is clear that the image of \cref{eq:interesting_collection} is a basis of the Grothendieck group $\sfK_0(X)$, thus the collection is of maximal length.
\begin{proof}[Proof of the \cref{thm:main_thm}]
	We first verify that the collection \cref{eq:interesting_collection} is exceptional. Since the collection is numerically exceptional and consists of sheaves, it suffices to check the vanishing of $\Hom$- and $\Ext^2$-spaces. Via Serre duality, the computation of an $\Ext^2$-space can be done by computing global sections of a divisor. Thus, abbreviating $\hom(-,-)=\dim \Hom(-,-)$ and $\ext^k(-,-)=\dim \Ext^k(-,-)$, we have to show that the following dimensions are zero:
	\begin{align*}
		\hom(\mcO_X(2F), \mcO_X(F))&=h^0(-F),\\
		\ext^2(\mcO_X(2F), \mcO_X(F))&=h^2(-F)=h^0(K_X+F),\\
		\hom(\mcO_X(2F), \mcO_X(D_i))&=h^0(D_i-2F),\\
		\ext^2(\mcO_X(2F), \mcO_X(D_i))&=h^2(D_i-2F)=h^0(K_X-D_i+2F),\\
		\hom(\mcO_X(2F), \mcO_X)&=h^0(-2F),\\
		\ext^2(\mcO_X(2F), \mcO_X)&=h^2(-2F)=h^0(K_X+2F),\\
		\hom(\mcO_X(F), \mcO_X(D_i))&=h^0(D_i-F),\\
		\ext^2(\mcO_X(F), \mcO_X(D_i))&=h^2(D_i-F)=h^0(K_X-D_i+F),\\
		\hom(\mcO_X(F), \mcO_X)&=h^0(-F),\\
		\ext^2(\mcO_X(F), \mcO_X)&=h^2(-F)=h^0(K_X+F),\\
		\hom(\mcO_X(D_i), \mcO_X(D_j))&=h^0(D_j-D_i),\\
		\ext^2(\mcO_X(D_i), \mcO_X(D_j))&=h^2(D_j-D_i)=h^0(K_X-D_j+D_i),\\
		\hom(\mcO_X(D_i), \mcO_X)&=h^0(-D_i),\\
		\ext^2(\mcO_X(D_i), \mcO_X)&=h^2(-D_i)=h^0(K_X+D_i),
	\end{align*}
	where $1\leq i, j \leq 10$, $i\neq j$.
	The vanishing holds trivially if the divisor has negative degree or is of the form $D_j-D_i=E_i-E_j$. The remaining cases are
	\begin{align}\label{eq:divisors_to_check}
		-F&=19H-6\sum_{j=1}^{10}E_j,\quad
		-2F=38H-12\sum_{j=1}^{10}E_j,\quad
		-D_i=6H-2\sum_{j=1}^{10}E_j +E_i, \\
		D_i-F&= 13H-4\sum_{j=1}^{10}E_j -E_i,\quad
		D_i-2F = 32H-10\sum_{j=1}^{10}E_j -E_i. \nonumber
	\end{align}
	Up to permutation of the points, these divisors are in standard form, see \cref{se:shgh_conj}. Thus by \cite[Prop.~1.4]{ciliberto_miranda_homogeneous_interpolation_on_ten_points}, if $D$ is one of the divisors in \cref{eq:divisors_to_check}, $C \cdot D \geq -1$ for all $(-1)$-curves $C \subseteq X$.
	If $D\neq -2F$, the multiplicities of $D$ are bounded by $11$ and thus the SHGH Conjecture is a theorem for $D$ by \cite{dumnicki_jarnicki_new_effective_bounds_on_the_dimension_of_a_linear_system_in_p2}. In the case $D=-2F$, the SHGH Conjecture is settled in \cite[Thm.~0.1]{ciliberto_miranda_homogeneous_interpolation_on_ten_points}. Thus, we obtain for each divisor $D$ in \cref{eq:divisors_to_check} that $h^0(D)= \chi(D)=0$.
	
	To show that \cref{eq:interesting_collection} is not full, by \cref{prop:height_zero_if_coll_full} and \cref{lem:height_bigger_than_pseudoheight}
	it suffices to show that the anticanonical pseudoheight of \cref{eq:interesting_collection} is at least $-1$.
	As the exceptional collection only consists of sheaves, it is enough to show that $\mathrm{e}(E_i, E_j) \neq 0$ for all $E_i, E_j$, $i<j$, appearing in \cref{eq:interesting_collection}. In other words, we have to show that the following dimensions vanish:
	\begin{align*}
		\hom(\mcO_X, \mcO_X(D_i))&=h^0(D_i),\\
		\hom(\mcO_X, \mcO_X(F))&=h^0(F),\\
		\hom(\mcO_X, \mcO_X(2F))&=h^0(2F),\\
		\hom(\mcO_X(D_i), \mcO_X(D_j))&=h^0(D_j-D_i),\\
		\hom(\mcO_X(D_i), \mcO_X(F))&=h^0(F-D_i),\\
		\hom(\mcO_X(D_i), \mcO_X(2F))&=h^0(2F-D_i),\\
		\hom(\mcO_X(F), \mcO_X(2F))&=h^0(F),
	\end{align*}
	where $1\leq i, j \leq 10$ and $i\neq j$.
	All these divisors have either negative degree or are of the form $D_j-D_i=E_i-E_j$, thus the vanishing holds for trivial reasons. Hence, $\mathrm{ph}_{\mathrm{ac}} >-2$ and we conclude that \cref{eq:interesting_collection} is not full.
\end{proof}
\begin{corollary}
	The admissible subcategory
	$$\mathcal{A}=\langle \mcO_X, \mcO_X(D_1), \dots, \mcO_X(D_{10}), \mcO_X(F), \mcO_X(2F) \rangle ^\perp$$
	is a universal phantom subcategory of $\sfD^b(X)$.
\end{corollary}
\begin{proof}
	The Chow motive of $X$ with integers coefficients is of Lefschetz type and $\sfK_0(\mathcal{A})=0$. By \cite[Cor.~4.3]{gorchinskiy_orlov_geometric_phantom_categories} the $\sfK$-motive of $\mathcal{A}$ with integer coefficients vanishes and \cite[Prop.~4.4]{gorchinskiy_orlov_geometric_phantom_categories} shows that $\mathcal{A}$ is a universal phantom.
\end{proof}
\begin{remark}
	Applying the techniques from \cite{kuznetsov_height_of_exceptional_collections_and_hochschild_cohomology_of_quasiphantom_categories} it is further possible to compute that the height of \cref{eq:interesting_collection} is $4$ and the Hochschild cohomology of $\mcA$ has the following dimensions:
	\begin{align*}
		\dim \mathrm{HH}^0(\mcA) &= 1,
		&&\dim \mathrm{HH}^1(\mcA) = 0, 
		&&\dim \mathrm{HH}^2(\mcA) = 12,
		&&\dim \mathrm{HH}^3(\mcA)  = 446,\\
		\dim \mathrm{HH}^4(\mcA)  &= 853,
		&&\dim \mathrm{HH}^5(\mcA)  = 420,
		&&\dim \mathrm{HH}^i(\mcA)  = 0 \;\, \text{for}\; i\geq 6.&&
	\end{align*}
\end{remark}

\printbibliography[title=References]

\end{document}